\documentclass[11pt]{amsart}
\usepackage{latexsym}
\usepackage{amssymb}
\usepackage{amscd}
\usepackage{float}
\usepackage{graphicx}
\usepackage{amsfonts}
\usepackage{pb-diagram}
 \usepackage{amsmath,amscd}
\usepackage{setspace}

\newtheorem{thm}{Theorem}

\newtheorem{lemma}{Lemma}
\newtheorem{prop}{Proposition}
\newtheorem{defn}{Definition}
\newtheorem{remark}{Remark}

\newtheorem{ex}{Example}
\newtheorem{nt}{Notation}

\begin{document}

\title[A step for $\mathcal{S}\left(L(p,1)\right)$ via braids]
  {An important step for the computation of the HOMFLYPT skein module of the lens spaces $L(p,1)$ via braids}

\author{Ioannis Diamantis}
\address{ International College Beijing,
China Agricultural University,
No.17 Qinghua East Road, Haidian District,
Beijing, {100083}, P. R. China.}
\email{ioannis.diamantis@hotmail.com}

\author{Sofia Lambropoulou}
\address{ Departament of Mathematics,
National Technical University of Athens,
Zografou campus,
{GR-15780} Athens, Greece.}
\email{sofia@math.ntua.gr}
\urladdr{http://www.math.ntua.gr/~sofia}



\setcounter{section}{-1}

\date{}

\begin{abstract}
We prove that, in order to derive the HOMFLYPT skein module of the lens spaces $L(p,1)$ from the HOMFLYPT skein module of the solid torus, $\mathcal{S}({\rm ST})$, it suffices to solve an infinite system of equations obtained by imposing on the Lambropoulou invariant $X$ for knots and links in the solid torus, braid band moves that are performed only on the first moving strand of elements in a set $\Lambda^{aug}$, augmenting the basis $\Lambda$ of $\mathcal{S}({\rm ST})$.
\end{abstract}

\maketitle

\section{Introduction and overview}\label{intro}

Skein modules generalize knot polynomials in $S^3$ to knot polynomials in arbitrary 3-manifolds \cite{Tu, P}. Skein modules are quotients of free modules over ambient isotopy classes of links in 3-manifolds by properly chosen local (skein) relations.

\begin{defn}\rm
Let $M$ be an oriented $3$-manifold, $R=\mathbb{Z}[u^{\pm1},z^{\pm1}]$, $\mathcal{L}$ the set of all oriented links in $M$ up to ambient isotopy in $M$ and let $S$ be the submodule of $R\mathcal{L}$ generated by the skein expressions $u^{-1}L_{+}-uL_{-}-zL_{0}$, where
$L_{+}$, $L_{-}$ and $L_{0}$ comprise a Conway triple represented schematically by the local illustrations in Figure~\ref{skein}.

\begin{figure}[!ht]
\begin{center}
\includegraphics[width=1.7in]{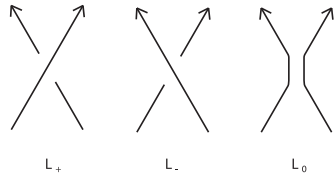}
\end{center}
\caption{The links $L_{+}, L_{-}, L_{0}$ locally.}
\label{skein}
\end{figure}

\noindent For convenience we allow the empty knot, $\emptyset$, and add the relation $u^{-1} \emptyset -u\emptyset =zT_{1}$, where $T_{1}$ denotes the trivial
knot. Then the {\it HOMFLYPT skein module} of $M$ is defined to be:

\begin{equation*}
\mathcal{S} \left(M\right)={\raise0.7ex\hbox{$
R\mathcal{L} $}\!\mathord{\left/ {\vphantom {R\mathcal{L} S }} \right. \kern-\nulldelimiterspace}\!\lower0.7ex\hbox{$ S  $}}=\mathcal{S} \left(M;{\mathbb Z}\left[u^{\pm 1} ,z^{\pm 1} \right],u^{-1} L_{+} -uL_{-} -zL{}_{0} \right).
\end{equation*}

\end{defn}

Note that the computation of $\mathcal{S}\left(M \right)$ is equivalent to constructing all possible independent analogues of the HOMFLYPT or 2-variable Jones polynomial for knots and links in $M$ and this is related to the rank of $\mathcal{S}\left(M \right)$.
\smallbreak

 In \cite{La2} the most generic analogue, $X$, of the HOMFLYPT polynomial for links in the solid torus $\rm ST$ has been derived from the generalized Hecke algebras of type $\rm B$, $\textrm{H}_{1,n}$, which is related to the knot theory of the solid torus and the Artin group of the Coxeter group of type B, $B_{1, n}$, via a unique Markov trace constructed on them. As explained in \cite{La2, DL2}, the Lambropoulou invariant $X$ recovers the HOMFLYPT skein module of ST, $\mathcal{S}({\rm ST})$, and is appropriate for extending the results to the lens spaces $L(p,1)$, since the combinatorial setting is the same as for $\rm ST$, only the braid equivalence includes the braid band moves (shorthanded to bbm), which reflect the surgery description of $L(p,1)$ via an unknotted surgery curve in $S^3$. Namely, in order to extend $X$ to an invariant of links in $L(p,1)$ we need to solve an infinite system of equations resulting from the braid band moves. Namely we force:

\begin{equation}\label{eqbbm}
X_{\widehat{\alpha}} =  X_{\widehat{bbm(\alpha)}},
\end{equation}

\noindent for all $\alpha \in \bigcup_{\infty}B_{1,n}$ and for all possible slidings of $\alpha$.

The above equations have particularly simple formulations with the use of a new basis, $\Lambda$, for the HOMFLYPT skein module of $\rm ST$, that we give in \cite{D, DL2} (see (\ref{basis}) in this paper) in terms of mixed braids (that is, classical braids with the first strand identically fixed and representing the complementary solid torus). Elements in $\Lambda$ consist of monomials in the $t_i$'s with consecutive indices and ordered exponents, while elements in the other basis, $\Lambda^{\prime}$, which first appeared in \cite{Tu, HK} in diagrammatic form, consist of monomials in the $t_i^{\prime}$'s with consecutive indices and ordered exponents (see Figure~\ref{genh}(ii)). In \cite{DL2} the set $\Lambda$ is related to the basis of $\mathcal{S}({\rm ST})$, $\Lambda^{\prime}$, via an infinite lower triangular matrix with invertible elements on the diagonal. Hence, this matrix allows the convertion of elements in $\Lambda^{\prime}$ to elements in $\Lambda$ and vice versa.

The bbm's are naturally described by elements in the basis $\Lambda$ and Equations~\ref{eqbbm} are very controlled in the algebraic setting, because, as shown in \cite{DLP}, they can be performed only on elements in $\Lambda$. The starting point in \cite{DLP} is the basis $\Lambda^{\prime}$ of $\mathcal{S}({\rm ST})$. Then, using conjugation and stabilization moves it is shown that the following diagram commutes:

\[
\begin{array}{ccccc}
\Lambda^{\prime} & \ni &  \tau_i^{\prime} & \overset{bbm_m}{\longrightarrow} & bbm_m(\tau_i^{\prime}) \\
                 &     &     {\uparrow}                &     &      {\uparrow}      \\
						     &     &     matrix                &     &      matrix      \\
                 &     &     {\downarrow}                &     &      {\downarrow}      \\
\Lambda          & \ni & \underset{i}{\sum} \tau_i &  \overset{bbm_m}{\longrightarrow} & \underset{i}{\sum} bbm_m(\tau_i) \\
\end{array}
\]

\noindent where $bbm_m(\tau_i^{\prime})$ denotes the result of the performance of a bbm on the $m^{th}$ moving strand of $\tau_i^{\prime}$. 

The fact that the $t_i^{\prime}$'s are conjugates on the braid level $B_{1, n}$, allows for a bbm performed on any moving strand to be translated, by conjugation, to bbm's performed on the first moving strand. This results in messing the order of the exponents in monomials of the $t_i^{\prime}$'s, so it leads naturally to introducing a new spanning set of $\mathcal{S}({\rm ST})$, ${\Lambda^{\prime}}^{aug}$, consisting of monomials in the $t_i^{\prime}$'s with consecutive indices but not necessarily ordered exponents, which obviously, augments $\Lambda^{\prime}$. Indeed, by conjugation, a bbm can be always assumed to take place on the first moving strand of elements in ${\Lambda^{\prime}}^{aug}$. Namely, the following diagram commutes:

\[
\begin{array}{ccccc}
\Lambda^{\prime} & \ni &  \tau_i^{\prime} & \overset{bbm_m}{\longrightarrow} & T_i^{\prime}  \\
                 &     &     \uparrow                &  &                              \uparrow  \\
								 &     &     conj.                    &  &                              |   \\
                 &     &     \downarrow                &  &                              \downarrow   \\
{\Lambda^{\prime}}^{aug} & \ni & \tau_j^{\prime} & \overset{bbm_1}{\longrightarrow} &  T_j^{\prime}  \\
\end{array}
\]

The set ${\Lambda^{\prime}}^{aug}$ corresponds naturally to the set $\Lambda^{aug}$ (see Eq.~\ref{lamaug}), which consists of monomials in the $t_i$'s with consecutive indices but not necessarily ordered exponents. 

\smallbreak

This paper is concerned with the analogue of the procedure above on the level of the $t_i$'s. In particular, we consider elements in the augmented set $\Lambda^{aug}$ and we restrict the performance of the braid band moves only on their {\it first moving strand}. Namely, we show that the following diagram commutes:

\[
\begin{array}{ccccc}
\Lambda & \ni &  \tau_i & \overset{bbm_m}{\longrightarrow} & bbm_m(\tau_i)  \\
                 &     &     \updownarrow                &  &                              \updownarrow  \\
{\Lambda}^{aug} & \ni & \underset{j}{\sum}\tau_j & \overset{bbm_1}{\longrightarrow} &  \underset{j}{\sum} bbm_1(\tau_j)  \\
\end{array}
\]

\noindent The fact that the $t_i$'s are not conjugate makes this procedure very non-trivial.

Our results are summarized in the commuting diagram below:

\[
\begin{array}{ccccccc}
{\bf \Lambda} & \ni &  \underset{i}{\sum} \tau_i &  & \overset{{\bf bbm_m}}{\longrightarrow} &  & \underset{i}{\sum} bbm_m(\tau_i) \\
        &     &     \underset{matrix}{\uparrow}                &     & &    \underset{[DLP]}{\nearrow}         &       {\uparrow}         \\
				        &     &     |                &     &  &             &        | \\
\Lambda^{\prime} & \ni &  \tau_i^{\prime} & \overset{bbm_m}{\longrightarrow} & bbm_m(\tau_i^{\prime}) & &| \\
                 &     &     \uparrow                &  &                              \uparrow  &      & | \\
								 &     &     conj.                    &  &                              |  &      & | \\
                 &     &     \downarrow                &  &                              \downarrow  &      & | \\
{\Lambda^{\prime}}^{aug} & \ni & \tau_j^{\prime} & \overset{bbm_1}{\longrightarrow} &  bbm_1(\tau_j^{\prime}) & &| \\
        &     &    |                &     &   &           &     |            \\
        &     &    \overset{matrix}{\downarrow}						&     &    & {\searrow}          &     \downarrow            \\
{\bf \Lambda^{aug}} & \ni & \underset{j}{\sum} \tau_j &  &  \overset{{\bf bbm_1}}{\longrightarrow} & & \underset{j}{\sum} bbm_1(\tau_j) \\
\end{array}
\]

In that way, a more controlled infinite system of equations is obtained using the generators $t_i$'s that are more natural for the bbm's.
Solving this infinite system is equivalent to computing the HOMFLYPT skein module of $L(p, 1)$. Namely:

\begin{equation}\label{result}
\mathcal{S}\left( L(p,1) \right) \ = \ \frac{\mathcal{S}({\rm ST})}{<\tau - bbm_1(\tau)>},\quad \tau\in \Lambda^{aug}.
\end{equation}

By our result, Equations~\ref{eqbbm} are now fewer. We demonstrate this by a simple example: For exponent sum equal to $3$ and all exponents positive, one has the following elements of $\Lambda$: $t^3, t^2t_1$ and $tt_1t_2$. From these elements one obtains $12$ equations by applying bbm's on all moving strands. On the other hand, $\Lambda^{aug}=\Lambda \bigcup \{tt_1^2\}$ and, by our result, one has to consider only $8$ equations. In that way we obtain more control on the infinite system (\ref{eqbbm}), even though $\Lambda \subset \Lambda^{aug}$. In \cite{DL4} we elaborate on the solution of this infinite system and compute the HOMFLYPT skein module of the lens spaces $L(p,1)$ via braids. We note  that in \cite{GM} $\mathcal{S}\left(L(p,1)\right)$ is computed using diagrammatic methods. The importance of our approach is that it can shed light on the problem of computing skein modules of arbitrary c.c.o. $3$-manifolds, since any $3$-manifold can be obtained by surgery on $S^3$ along unknotted closed curves. The main difficulty of the problem lies in selecting from the infinitum of band moves some basic ones and solving the infinite system of equations.

\bigbreak

The paper is organized as follows: In \S\ref{basics} we recall the setting and the essential techniques and results from \cite{La1, La2, LR1, LR2, DL1, DL2}. More precisely, we first present isotopy moves for knots and links in $L(p,1)$ and we describe the braid equivalence for knots and links in $L(p,1)$. We then present the new basis $\Lambda$ for $\mathcal{S}({\rm ST})$, with the use of which the braid band moves are naturally described, and we recall from \cite{DL2} the ordering defined on the sets $\Lambda$ and $\Lambda^{aug}$. We then recall from \cite{DLP, DL3} the derivation of $\mathcal{S}\left(L(p,1) \right)$ from $\mathcal{S}({\rm ST})$ using the braid approach. In \S~2 we prove the main result of this paper, Theorem~\ref{mainre}. That is, we show that in order to compute $\mathcal{S}\left(L(p,1) \right)$ from $\mathcal{S}({\rm ST})$ we only need to consider braid band moves on the first moving strands of elements in the set $\Lambda^{aug}$ and solve the infinite system of equations derived by imposing the Lambropoulou invariant $X$. We prove Theorem~\ref{mainre} by strong induction on the order of elements in $\Lambda^{aug}$ and on the moving strand where the bbm is performed. We first present a series of lemmata demonstrating how conjugation and stabilization moves can be used in order to convert elements in the set $\Lambda$ into sums of elements in $\Lambda^{aug}$ of lower order. The basis of the induction concerns monomials in $\Lambda$ of index 1 and it is proved in Proposition~\ref{indba}. Finally, in Proposition~\ref{prop}, using results from \cite{DL2, DLP}, we prove that equations obtained from elements in $\Lambda$ by performing bbm's on their $m^{th}$-moving strand are equivalent to equations obtained from elements in $\Lambda^{aug}$ of lower order by performing bbm's on their $j^{th}$-moving strand, where $j<m$. Using the above result we then conclude the proof of Theorem~\ref{mainre}.

\section{Topological and algebraic background}\label{basics}

\subsection{Mixed links and isotopy in $L(p,1)$}
 
We consider ST to be the complement of a solid torus in $S^3$. As explained in \cite{LR1, LR2, DL1}, an oriented link $L$ in ST can be represented by an oriented \textit{mixed link} in $S^{3}$, that is a link in $S^{3}$ consisting of the unknotted fixed part $\widehat{I}$ representing the complementary solid torus in $S^3$ and the moving part $L$ that links with $\widehat{I}$. A \textit{mixed link diagram} is a diagram $\widehat{I}\cup \widetilde{L}$ of $\widehat{I}\cup L$ on the plane of $\widehat{I}$, where this plane is equipped with the top-to-bottom direction of $I$ (see top left hand side of Figure~\ref{bmov}).

\smallbreak

The lens spaces $L(p,1)$ can be obtained from $S^3$ by surgery on the unknot with surgery coefficient $p\in \mathbb{Z}$. Surgery along the unknot can be realized by considering first the complementary solid torus and then attaching to it a solid torus according to some homeomorphism on the boundary. Thus, isotopy in $L(p,1)$ can be viewed as isotopy in ST together with the band moves in $S^3$, which reflect the surgery description of the manifold. Moreover, in \cite{DL1} it is shown that in order to describe isotopy for knots and links in a c.c.o. $3$-manifold, it suffices to consider only the type $a$ band moves (for an illustration see top of Figure~\ref{bmov}) and thus, isotopy between oriented links in $L(p,1)$ is reflected in $S^3$ by means of the following result (cf. Thm.~5.8 \cite{LR1}, Thm.~6 \cite{DL1} ):

\smallbreak

{\it
Two oriented links in $L(p,1)$ are isotopic if and only if two corresponding mixed link diagrams of theirs differ by isotopy in {\rm ST} together with a finite sequence of the type $a$ band moves.
}

\subsection{Mixed braids and braid equivalence for knots and links in $L(p,1)$}

By the Alexander theorem for knots and links in the solid torus (cf. Thm.~1 \cite{La2}), a mixed link diagram $\widehat{I}\cup \widetilde{L}$ of $\widehat{I}\cup L$ may be turned into a \textit{mixed braid} $I\cup \beta$ with isotopic closure. This is a braid in $S^{3}$ where, without loss of generality, its first strand represents $\widehat{I}$, the fixed part, and the other strands, $\beta$, represent the moving part $L$. The subbraid $\beta$ is called the \textit{moving part} of $I\cup \beta$ (see bottom left hand side of Figure~\ref{bmov}). Then, in order to translate isotopy for links in $L(p,1)$ into braid equivalence, we first perform the technique of {\it standard parting} introduced in \cite{LR2} in order to separate the moving strands from the fixed strand that represents the lens spaces $L(p,1)$. This can be realized by pulling each pair of corresponding moving strands to the right and {\it over\/} or {\it under\/} the fixed strand that lies on their right. Then, we define a {\it braid band move} to be a move between mixed braids, which is a band move between their closures. It starts with a little band oriented downward, which, before sliding along a surgery strand, gets one twist {\it positive\/} or {\it negative\/} (see bottom of Figure~\ref{bmov}).

\begin{figure}
\begin{center}
\includegraphics[width=3.4in]{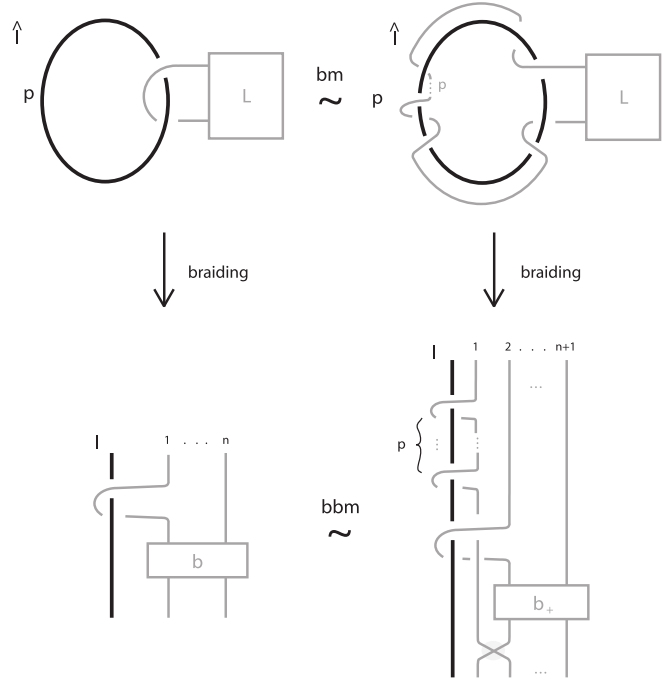}
\end{center}
\caption{Isotopy in $L(p,1)$ and the two types of braid band moves on mixed braids.}
\label{bmov}
\end{figure}

\smallbreak

The sets of braids related to ST form groups, which are in fact the Artin braid groups of type B, denoted $B_{1,n}$, with presentation:

\[ B_{1,n} = \left< \begin{array}{ll}  \begin{array}{l} t, \sigma_{1}, \ldots ,\sigma_{n-1}  \\ \end{array} & \left| \begin{array}{l}
\sigma_{1}t\sigma_{1}t=t\sigma_{1}t\sigma_{1} \ \   \\
 t\sigma_{i}=\sigma_{i}t, \quad{i>1}  \\
{\sigma_i}\sigma_{i+1}{\sigma_i}=\sigma_{i+1}{\sigma_i}\sigma_{i+1}, \quad{ 1 \leq i \leq n-2}   \\
 {\sigma_i}{\sigma_j}={\sigma_j}{\sigma_i}, \quad{|i-j|>1}  \\
\end{array} \right.  \end{array} \right>, \]

\noindent where the generators $\sigma _{i}$ and $t$ are illustrated in Figure~\ref{genh}(i).


Let now $\mathcal{L}$ denote the set of oriented knots and links in ST. Then, isotopy in $L(p,1)$ is then translated on the level of mixed braids by means of the following theorem:

\begin{thm}[Theorem~5, \cite{LR2}] \label{markov}
 Let $L_{1} ,L_{2}$ be two oriented links in $L(p,1)$ and let $I\cup \beta_{1} ,{\rm \; }I\cup \beta_{2}$ be two corresponding mixed braids in $S^{3}$. Then $L_{1}$ is isotopic to $L_{2}$ in $L(p,1)$ if and only if $I\cup \beta_{1}$ is equivalent to $I\cup \beta_{2}$ in $\mathcal{B}$ by the following moves:
\[ \begin{array}{clll}
(i)  & Conjugation:         & \alpha \sim \beta^{-1} \alpha \beta, & {\rm if}\ \alpha ,\beta \in B_{1,n}. \\
(ii) & Stabilization\ moves: &  \alpha \sim \alpha \sigma_{n}^{\pm 1} \in B_{1,n+1}, & {\rm if}\ \alpha \in B_{1,n}. \\
(iii) & Loop\ conjugation: & \alpha \sim t^{\pm 1} \alpha t^{\mp 1}, & {\rm if}\ \alpha \in B_{1,n}. \\
(iv) & Braid\ band\ moves: & \alpha \sim {t}^p \alpha_+ \sigma_1^{\pm 1}, & a_+\in B_{1, n+1},
\end{array} \]

\noindent where $\alpha_+$ is the word $\alpha$ with all indices shifted by +1. Note that moves (i), (ii) and (iii) correspond to link isotopy in {\rm ST}.
\end{thm}

\begin{nt}\rm
We denote a braid band move by bbm and, specifically, the result of a positive or negative braid band move performed on the $i^{th}$-moving strand of a mixed braid $\beta$ by $bbm_{\pm i}(\beta)$.
\end{nt}

Note also that in \cite{LR2} it was shown that the choice of the position of connecting the two components after the performance of a bbm is arbitrary.

\subsection{The HOMFLYPT skein module of ST via braids}\label{SolidTorus}

In \cite{La2} the most generic analogue of the HOMFLYPT polynomial, $X$, for links in the solid torus $\rm ST$ has been derived from the generalized Iwahori--Hecke algebras of type $\rm B$, $\textrm{H}_{1,n}$, via a unique Markov trace constructed on them. This algebra was defined by the second author as the quotient of ${\mathbb C}\left[q^{\pm 1} \right]B_{1,n}$ over the quadratic relations ${g_{i}^2=(q-1)g_{i}+q}$. Namely:

\begin{equation*}
\textrm{H}_{1,n}(q)= \frac{{\mathbb C}\left[q^{\pm 1} \right]B_{1,n}}{ \langle \sigma_i^2 -\left(q-1\right)\sigma_i-q \rangle}.
\end{equation*}

It is also shown that the following sets form linear bases for ${\rm H}_{1,n}(q)$ (\cite[Proposition~1 \& Theorem~1]{La2}):

\[
\begin{array}{llll}
 (i) & \Sigma_{n} & = & \{t_{i_{1} } ^{k_{1} } \ldots t_{i_{r}}^{k_{r} } \cdot \sigma \} ,\ {\rm where}\ 0\le i_{1} <\ldots <i_{r} \le n-1,\\
 (ii) & \Sigma^{\prime} _{n} & = & \{ {t^{\prime}_{i_1}}^{k_{1}} \ldots {t^{\prime}_{i_r}}^{k_{r}} \cdot \sigma \} ,\ {\rm where}\ 0\le i_{1} < \ldots <i_{r} \le n-1, \\
\end{array}
\]
\noindent where $k_{1}, \ldots ,k_{r} \in {\mathbb Z}$, $t_0^{\prime}\ =\ t_0\ :=\ t, \quad t_i^{\prime}\ =\ g_i\ldots g_1tg_1^{-1}\ldots g_i^{-1} \quad {\rm and}\quad t_i\ =\ g_i\ldots g_1tg_1\ldots g_i$ are the `looping elements' in ${\rm H}_{1, n}(q)$ (see Figure~\ref{genh}(ii)) and $\sigma$ a basic element in the Iwahori--Hecke algebra of type A, ${\rm H}_{n}(q)$, for example in the form of the elements in the set \cite{Jo}:

$$ S_n =\left\{(g_{i_{1} }g_{i_{1}-1}\ldots g_{i_{1}-k_{1}})(g_{i_{2} }g_{i_{2}-1 }\ldots g_{i_{2}-k_{2}})\ldots (g_{i_{p} }g_{i_{p}-1 }\ldots g_{i_{p}-k_{p}})\right\}, $$

\noindent for $1\le i_{1}<\ldots <i_{p} \le n-1{\rm \; }$. In \cite{La2} the bases $\Sigma^{\prime}_{n}$ are used for constructing a Markov trace on $\mathcal{H}:=\bigcup_{n=1}^{\infty}{\rm H}_{1, n}$, and using this trace the second author constructed a universal HOMFLYPT-type invariant for oriented links in ST.


\begin{figure}\label{loopttpr}
\begin{center}
\includegraphics[width=5.1in]{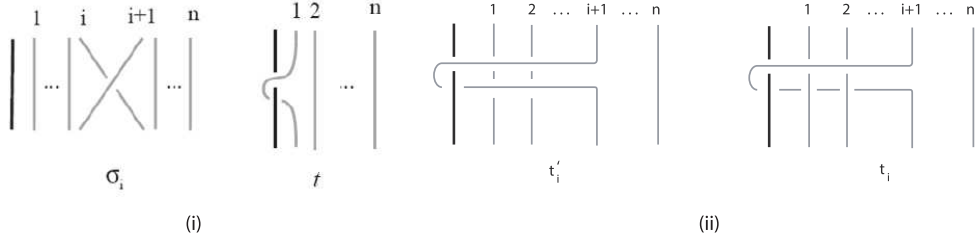}
\end{center}
\caption{The generators of $B_{1, n}$ and the `looping' elements $t^{\prime}_{i}$ and $t_{i}$.}
\label{genh}
\end{figure}

\begin{thm}{\cite[Theorem~6 \& Definition~1]{La2}} \label{tr}
Given $z, s_{k}$ with $k\in {\mathbb Z}$ specified elements in $R={\mathbb C}\left[q^{\pm 1} \right]$, there exists a unique linear Markov trace function on $\mathcal{H}$:

\begin{equation*}
{\rm tr}:\mathcal{H}  \to R\left(z,s_{k} \right),\ k\in {\mathbb Z}
\end{equation*}

\noindent determined by the rules:

\[
\begin{array}{lllll}
(1) & {\rm tr}(ab) & = & {\rm tr}(ba) & \quad {\rm for}\ a,b \in {\rm H}_{1,n}(q) \\
(2) & {\rm tr}(1) & = & 1 & \quad {\rm for\ all}\ {\rm H}_{1,n}(q) \\
(3) & {\rm tr}(ag_{n}) & = & z{\rm tr}(a) & \quad {\rm for}\ a \in {\rm H}_{1,n}(q) \\
(4) & {\rm tr}(a{t^{\prime}_{n}}^{k}) & = & s_{k}{\rm tr}(a) & \quad {\rm for}\ a \in {\rm H}_{1,n}(q),\ k \in {\mathbb Z} \\
\end{array}
\]

\bigbreak

\noindent Then, the function $X:\mathcal{L}$ $\rightarrow R(z,s_{k})$

\begin{equation*}
X_{\widehat{\alpha}} = \Delta^{n-1}\cdot \left(\sqrt{\lambda } \right)^{e}
{\rm tr}\left(\pi \left(\alpha \right)\right),
\end{equation*}

\noindent is an invariant of oriented links in {\rm ST}, where $\Delta:=-\frac{1-\lambda q}{\sqrt{\lambda } \left(1-q\right)}$, $\lambda := \frac{z+1-q}{qz}$, $\alpha \in B_{1,n}$ is a word in the $\sigma _{i}$'s and $t^{\prime}_{i} $'s, $\widehat{\alpha}$ is the closure of $\alpha$, $e$ is the exponent sum of the $\sigma _{i}$'s in $\alpha $, $\pi$ the canonical map of $B_{1,n}$ on ${\rm H}_{1,n}(q)$, such that $t\mapsto t$ and $\sigma _{i} \mapsto g_{i}$.
\end{thm}

\bigbreak

In the braid setting of \cite{La2}, the elements of $\mathcal{S}({\rm ST})$ correspond bijectively to the elements of the following set $\Lambda^{\prime}$:

\begin{equation}\label{Lpr}
\Lambda^{\prime}=\{ {t^{k_0}}{t^{\prime}_1}^{k_1} \ldots
{t^{\prime}_n}^{k_n}, \ k_i \in \mathbb{Z}\setminus\{0\}, \ k_i \leq k_{i+1}\ \forall i,\ n\in \mathbb{N} \}.
\end{equation}

\noindent As explained in \cite{La2, DL2}, the set $\Lambda^{\prime}$ forms a basis of $\mathcal{S}({\rm ST})$ in terms of braids (see also \cite{HK, Tu}). Note that $\Lambda^{\prime}$ is a subset of $\mathcal{H}$ and, in particular, $\Lambda^{\prime}$ is a subset of $\Sigma^{\prime}=\bigcup_n\Sigma^{\prime}_n$. Note also that in contrast to elements in $\Sigma^{\prime}$, the elements in $\Lambda^{\prime}$ have no gaps in the indices, the exponents are ordered and there are no `braiding tails'. 

\begin{remark}\rm
The Lambropoulou invariant $X$ recovers $\mathcal{S}({\rm ST})$. Indeed, it gives distinct values to distinct elements of $\Lambda^{\prime}$, since ${\rm tr}(t^{k_0}{t^{\prime}_1}^{k_1} \ldots {t^{\prime}_n}^{k_n})=s_{k_n}\ldots s_{k_1}s_{k_0}$.
\end{remark}

\subsection{A different basis for $\mathcal{S}({\rm ST})$}

In \cite{DL2}, a different basis $\Lambda$ for $\mathcal{S}({\rm ST})$ is presented, which is crucial toward the computation of $\mathcal{S}\left(L(p,1)\right)$ and which is described in Eq.~\ref{basis} in open braid form (for an illustration see Figure~\ref{basel}). In particular we have the following:

\begin{thm}{\cite[Theorem~2]{DL2}}\label{newbasis}
The following set is a $\mathbb{C}[q^{\pm1}, z^{\pm1}]$-basis for $\mathcal{S}({\rm ST})$:
\begin{equation}\label{basis}
\Lambda=\{t^{k_0}t_1^{k_1}\ldots t_n^{k_n},\ k_i \in \mathbb{Z}\setminus\{0\},\ k_i \leq k_{i+1}\ \forall i,\ n \in \mathbb{N} \}.
\end{equation}
\end{thm}

The importance of the new basis $\Lambda$ of $\mathcal{S}({\rm ST})$ lies in the simplicity of the algebraic expression of a braid band move (recall Theorem~1(iv)), which extends the link isotopy in ST to link isotopy in $L(p,1)$ and this fact was our motivation for establishing this new basis $\Lambda$. Note that comparing the set $\Lambda$ with the set $\Sigma=\bigcup_n\Sigma_n$, we observe that in $\Lambda$ there are no gaps in the indices of the $t_i$'s and the exponents are in decreasing order. Also, there are no `braiding tails' in the words in $\Lambda$.

\begin{figure}
\begin{center}
\includegraphics[width=1.6in]{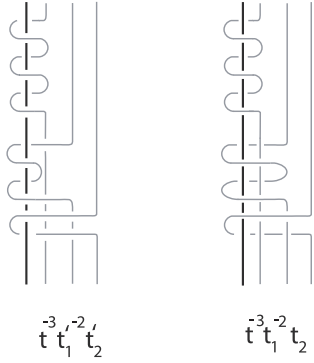}
\end{center}
\caption{Elements in the two different bases of $\mathcal{S}({\rm ST})$.}
\label{basel}
\end{figure}

\subsection{An ordering in the bases of $\mathcal{S}({\rm ST})$}

We now proceed with defining an ordering relation in the sets $\Sigma$ and $\Sigma^{\prime}$, which passes to their respective subsets $\Lambda$ and $\Lambda^{\prime}$ and that first appeared in \cite{DL2}. This ordering relation plays a crucial role to what will follow. For that we need the notion of the {\it index} of a word $w$ in any of these sets, denoted $ind(w)$. In $\Lambda^{\prime}$ or $\Lambda$ $ind(w)$ is defined to be the highest index of the $t_i^{\prime}$'s, resp. of the $t_i$'s in $w$. Similarly, in $\Sigma^{\prime}$ or $\Sigma$, $ind(w)$ is defined as above by ignoring possible gaps in the indices of the looping generators and by ignoring the braiding parts in the algebras $\textrm{H}_{n}(q)$. Moreover, the index of a monomial in $\textrm{H}_{n}(q)$ is equal to $0$.

\begin{defn}{\cite[Definition~2]{DL2}} \label{order}
\rm
Let $w={t^{\prime}_{i_1}}^{k_1}\ldots {t^{\prime}_{i_{\mu}}}^{k_{\mu}}\cdot \beta_1$ and $u={t^{\prime}_{j_1}}^{\lambda_1}\ldots {t^{\prime}_{j_{\nu}}}^{\lambda_{\nu}}\cdot \beta_2$ in $\Sigma^{\prime}$, where $k_t , \lambda_s \in \mathbb{Z}$ for all $t,s$ and $\beta_1, \beta_2 \in H_n(q)$. Then, we define the following ordering in $\Sigma^{\prime}$:

\smallbreak

\begin{itemize}
\item[(a)] If $\sum_{i=0}^{\mu}k_i < \sum_{i=0}^{\nu}\lambda_i$, then $w<u$.

\vspace{.1in}

\item[(b)] If $\sum_{i=0}^{\mu}k_i = \sum_{i=0}^{\nu}\lambda_i$, then:

\vspace{.1in}

\noindent  (i) if $ind(w)<ind(u)$, then $w<u$,

\vspace{.1in}

\noindent  (ii) if $ind(w)=ind(u)$, then:

\vspace{.1in}

\noindent \ \ \ \ ($\alpha$) if $i_1=j_1, \ldots , i_{s-1}=j_{s-1}, i_{s}<j_{s}$, then $w>u$,

\vspace{.1in}

\noindent \ \ \  ($\beta$) if $i_t=j_t$ for all $t$ and $k_{\mu}=\lambda_{\mu}, k_{\mu-1}=\lambda_{\mu-1}, \ldots, k_{i+1}=\lambda_{i+1}, |k_i|<|\lambda_i|$, then $w<u$,

\vspace{.1in}

\noindent \ \ \  ($\gamma$) if $i_t=j_t$ for all $t$ and $k_{\mu}=\lambda_{\mu}, k_{\mu-1}=\lambda_{\mu-1}, \ldots, k_{i+1}=\lambda_{i+1}, |k_i|=|\lambda_i|$ and $k_i>\lambda_i$, then $w<u$,

\vspace{.1in}

\noindent \ \ \ \ ($\delta$) if $i_t=j_t\ \forall t$ and $k_i=\lambda_i$, $\forall i$, then $w=u$.

\end{itemize}

The ordering in the set $\Sigma$ is defined as in $\Sigma^{\prime}$, where $t_i^{\prime}$'s are replaced by $t_i$'s.
\end{defn}


\bigbreak

 We also define the \textit{subsets of level $k$}, $\Lambda_{(k)}$ and $\Lambda^{\prime}_{(k)}$, of $\Lambda$ and $\Lambda^{\prime}$ respectively ({\cite[Definition~3]{DL2}}), to be the sets:

\begin{equation}
\begin{array}{l}
\Lambda_{(k)}:=\{t_0^{k_0}t_1^{k_1}\ldots t_{m}^{k_m} | \sum_{i=0}^{m}{k_i}=k,\ k_i \in \mathbb{Z}\setminus\{0\},\  k_i \leq k_{i+1}\ \forall i \}\\
\\
\Lambda^{\prime}_{(k)}:=\{{t^{\prime}_0}^{k_0}{t^{\prime}_1}^{k_1}\ldots {t^{\prime}_{m}}^{k_m} | \sum_{i=0}^{m}{k_i}=k,\ k_i \in \mathbb{Z}\setminus\{0\},\  k_i \leq k_{i+1}\ \forall i \}
\end{array}
\end{equation}

\noindent In \cite{DL2} it was shown that the sets $\Lambda_{(k)}$ and $\Lambda^{\prime}_{(k)}$ are totally ordered and well ordered for all $k$ (\cite[Propositions~1 \& 2]{DL2}). Note that in \cite{DLP} the exponents in the monomials of $\Lambda$ are in decreasing order, while here the exponents are considered in increasing order, which is totally symmetric. 

\smallbreak

We finally define the set $\Lambda^{aug}$ which augments the basis $\Lambda$ and its subset of level $k$:

\begin{defn}\rm
We define the set:
\begin{equation}\label{lamaug}
\Lambda^{aug}\ :=\{t_0^{k_0}t_1^{k_1}\ldots t_{n}^{k_n},\ k_i \in \mathbb{Z}\backslash \{0\}\}.
\end{equation}
\noindent and the {\it subset of level} $k$, $\Lambda^{aug}_{(k)}$, of $\Lambda^{aug}$:

\begin{equation}\label{lamaug2}
\Lambda^{aug}_{(k)}:=\{t_0^{k_0}t_1^{k_1}\ldots t_{m}^{k_m} | \sum_{i=0}^{m}{k_i}=k,\ k_i \in \mathbb{Z}\backslash \{0\}\}
\end{equation}
\end{defn}

\section{Restricting the bbm's on the first moving strand of $\Lambda^{aug}$}

As mentioned in the Introduction, in order to compute $\mathcal{S}(L(p,1))$ we need to normalize the invariant $X$ by forcing it to satisfy all possible braid band moves. Namely:

$$
\mathcal{S}\left( L(p,1) \right) \ = \ \frac{\mathcal{S}({\rm ST})}{<a - bbm_i(a)>},\quad {\rm for\ all}\ a\in B_{1, n}\ {\rm and\ for\ all}\ i.
$$

\subsection{Applying bbm's on any moving strand of $\Lambda$}

In order to simplify this system of equations, in \cite{DLP} we first show that performing a bbm on a mixed braid in $B_{1, n}$ reduces to performing bbm's on elements in the canonical basis, $\Sigma_n^{\prime}$, of the algebra ${\rm H}_{1,n}(q)$ and, in fact, on their first moving strand. We then reduce the equations obtained from elements in $\Sigma^{\prime}$ to equations obtained from elements in $\Sigma$. In order now to reduce further the computation to elements in the basis $\Lambda$ of $\mathcal{S}({\rm ST})$, in \cite{DLP} we first recall that elements in $\Sigma$ consist in two parts: a monomial in $t_i$'s with possible gaps in the indices and unordered exponents, followed by a `braiding tail' in the basis of ${\rm H}_n(q)$. So, in \cite{DLP}, we first manage the gaps in the indices of the looping generators of elements in $\Sigma$, obtaining elements in the augmented ${\rm H}_{n}(q)$-module $\Lambda^{aug}$ (that is, elements in $\Lambda^{aug}$ followed by `braiding tails'). We denote the ${\rm H}_{n}(q)$-module $\Lambda^{aug}$ by $\Lambda^{aug}|{\rm H}_n$. Note that the procedure of managing the gaps forces the performance of bbm's to take place on {\it any moving} strand. We then show in \cite{DLP} that the equations obtained from elements in $\Lambda^{aug}|{\rm H}_n$ by performing bbm's on any moving strand are equivalent to equations obtained from elements in the ${\rm H}_{n}(q)$-module $\Lambda$, denoted by $\Lambda|{\rm H}_n$, by performing bbm's on any moving strand. By this procedure we order the exponents of the $t_i$'s (`ordering the exponents'). We finally eliminate the `braiding tails' from elements in $\Lambda|{\rm H}_n$ and reduce the computations to the set $\Lambda$, where the bbm's are performed on any moving strand (see \cite{DLP}). Thus, in order to compute $\mathcal{S}(L(p,1))$, it suffices to solve the infinite system of equations obtained by performing bbm's on any moving strand of elements in the set $\Lambda$.

\smallbreak

The above are summarized in the following sequence of equations:

$$
\begin{array}{llllll}
\mathcal{S}\left( L(p,1) \right) & = & \frac{\mathcal{S}({\rm ST})}{<a - bbm_i(a)>},\ a\in B_{1, n}, \ \forall\ i & = & \frac{\mathcal{S}({\rm ST})}{<s^{\prime} - bbm_1(s^{\prime})>},\ s^{\prime}\in \Sigma_n^{\prime} & = \\
&&&&&\\
& = & \frac{\mathcal{S}({\rm ST})}{<s - bbm_1(s)>},\ s\in \Sigma_n & = & \frac{\mathcal{S}({\rm ST})}{<\lambda^{\prime} - bbm_i(\lambda^{\prime})>},\ \lambda^{\prime} \in \Lambda^{aug}|{\rm H}_n,\ \forall\ i & = \\
&&&&&\\
& = & \frac{\mathcal{S}({\rm ST})}{<\lambda^{\prime \prime} - bbm_i(\lambda^{\prime \prime})>},\ \lambda^{\prime \prime}\in \Lambda|{\rm H}_n, \ \forall\ i & = & \frac{\mathcal{S}({\rm ST})}{<\lambda - bbm_i(\lambda)>},\ \lambda\in \Lambda, \ \forall\ i.\\
\end{array}
$$

Namely, we have:

\begin{thm}\label{dlpl}{\cite{DLP}}
$\mathcal{S}\left( L(p,1) \right)\ =\ \frac{\mathcal{S}({\rm ST})}{<\lambda - bbm_i(\lambda)>},\ \lambda\in \Lambda, \ \forall\ i.$
\end{thm}

In this section we prove Eq.~\ref{result}, assuming Theorem~\ref{dlpl}. More precisely, we consider the augmented set $\Lambda^{aug}$ and show that the system of equations obtained from elements in $\Lambda$ by performing bbm's on {\it any moving} strand, is equivalent to the system of equations obtained by performing bbm's on the {\it first moving} strand of elements in $\Lambda^{aug}$. It is worth mentioning that although $\Lambda^{aug} \supset \Lambda$, the advantage of considering elements in the augmented set $\Lambda^{aug}$ is that we restrict the performance of the braid band moves only on the first moving strand and, thus, we obtain less equations and more control on the infinite system (\ref{result}).

\subsection{The main result}

We now proceed with stating the main result of this paper.

\begin{thm}\label{mainre}
Equations obtained by performing braid band moves on the $m^{th}$-moving strand on an element in $\Lambda$ are combinations of equations obtained by performing braid band moves on the $1^{st}$-moving strand on elements in $\Lambda^{aug}$ of lower order. Namely, the following diagram commutes for all $k$ and for $a_i\in \mathbb{C}\left[q^{\pm 1}, z^{\pm 1} \right]$:

\[
\begin{array}{cccccc}
\Lambda_{(k)} & \ni & T & \overset{bbm}{\underset{m^{th}-mov. str.}{\longrightarrow}} & t^pT_{+}\sigma_m\ldots \sigma_2 \sigma_1^{\epsilon}\sigma^{-2}\ldots \sigma_m^{-1}& \\
& & | & & |& \\
& & {conj.}\ \& \ {stab.} & & {conj.}\ \& \ {stab.}& \\
& & \downarrow & & \downarrow & \\
\Lambda^{aug}_{(k)} & \ni & \underset{i}{\sum}a_i T_i & \overset{bbm}{\underset{1^{st}-mov. str.}{\longrightarrow}} & \underset{i}{\sum}a_i t^pT_{i_{+}}  \sigma_1^{\epsilon} & \in\ \Lambda^{aug}_{(p+k)}|H_n\\
\end{array}
\]
\end{thm}

\noindent The above diagram is illustrated in Figure~\ref{thmm}. We prove Theorem~\ref{mainre} by strong induction on the order of elements in $\Lambda^{aug}$ and on the moving strand where the bbm is performed. The basis of the induction concerns elements in $\Lambda$ of the form $t^{k_0} t_1^{k_1}$, where $k_0, k_1 \in \mathbb{Z}$. We now present an example illustrating Theorem~\ref{mainre}.

\begin{ex}\rm
Let $tt_1^2\in \Lambda_{(3)}$. Performing a positive braid band move on the second moving strand of $tt_1^2$ we obtain $t^pt_1t_2^2\cdot \sigma_2\sigma_1\sigma_2^{-1}$. We now compute ${\rm tr}(tt_1^2)$ and ${\rm tr}(t^pt_1t_2^2\cdot g_2g_1g_2^{-1})$:

\begin{equation}\label{eql0}
\begin{array}{rll}
{\rm tr}(tt_1^2) & = & (q^2-q+1)\cdot {\rm tr}(t_1^2t_2)\ +\ q(q-1)z\cdot {\rm tr}(t_1^3),\\
&&\\
{\rm tr}(t^pt_1t_2^2\cdot g_2g_1g_2^{-1}) & = & (q^2-q+1)\cdot {\rm tr}(t^pt_1^2t_2\cdot g_1)\ +\ q(q-1)z\cdot {\rm tr}(t^pt_1^3g_1)\ \quad {\rm and\ so}\\
\end{array}
\end{equation}

\begin{equation}\label{eql1}
X_{\widehat{tt_1^2}}\ =\ X_{\widehat{t^pt_1t_2^2\cdot \sigma_2\sigma_1\sigma_2^{-1}}}\Leftrightarrow {\rm tr}(tt_1^2)\ =\ - \frac{\lambda^3}{z}\cdot (q^2-q+1){\rm tr}(t^pt_1^2t_2g_1)\ -\ \lambda^3\cdot q(q-1){\rm tr}(t^pt_1^3g_1).
\end{equation}

\smallbreak

\noindent Consider now the elements $t^3, t^2t_1 \in \Lambda^{aug}_{(3)}$ and perform a positive braid band move on their first moving strand. We have that:

\begin{equation}\label{eql2}
\begin{array}{ccccccc}
X_{\widehat{t^3}} & = & X_{\widehat{t^pt_1^3\sigma_1}} & \Rightarrow & {\rm tr}(t^3) & = & - \frac{\lambda^3}{z}\cdot {\rm tr}(t^pt_1^3g_1)\\
&&&&&&\\
X_{\widehat{t^2t_1}} & = & X_{\widehat{t^pt_1^2t_2\sigma_1}} & \Rightarrow & {\rm tr}(t^2t_1) & = & - \frac{\lambda^3}{z}\cdot {\rm tr}(t^pt_1^2t_2g_1).\\
\end{array}
\end{equation}
 
\noindent Substituting to Eq.~\ref{eql1}, we obtain: 

$$
{\rm tr}(tt_1^2)\ =\ (q^2-q+1)\cdot {\rm tr}(t^2t_1)\ -\ q(q-1)z\cdot {\rm tr}(t^3),
$$

\noindent which is true (Equation~\ref{eql0}). Thus, we have shown that Equation~\ref{eql1} is a combination of Equations~\ref{eql2}, that is, the equation obtained from $tt_1^2\in \Lambda_{(3)}$ by performing bbm on its second moving strand is equivalent to equations obtained from $t^3, t^2t_1 \in \Lambda^{aug}_{(3)}$ by performing bbm's on their first moving strand.
\end{ex}

\begin{figure}
\begin{center}
\includegraphics[width=2.9in]{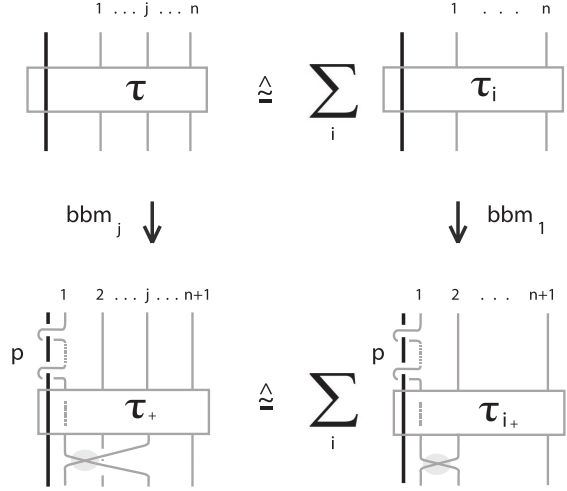}
\end{center}
\caption{An illustration of Theorem~\ref{mainre}.}
\label{thmm}
\end{figure}

\subsection{Some useful lemmata}

We now present some useful results which are crucial for the proof of Theorem~\ref{mainre}. We start by expressing the trace of the element $tt_1^k \in \Lambda_{(k+1)}^{aug}$ as a linear combination of traces of lower order elements in $\Lambda_{(k+1)}^{aug}$. We first recall the following equalities from \cite{La2}:

\begin{equation}\label{l2}
t_n^kg_n \ = \ sign(k) \cdot (q-1)\sum_{j=0}^{k-1}{q^jt_{n-1}^jt_n^{k-j}}+q^kg_nt_{n-1}^k,\quad {\rm where}\ k\in \mathbb{Z}.
\end{equation}

\begin{lemma}\label{old2}
The following relations hold in ${\rm H}_{1, n}(q)$:
\[
\begin{array}{lllll}
(i) & {\rm For}\ k \in \mathbb{N}: & t_n^kg_{n+1} & = & q^{-(k-1)}g_{n+1}^{-1} t_{n+1}^k \ + \ (q^{-1}-1)\underset{j=0}{\overset{k-2}{\sum}}{q^{-j}t_{n}^{k-j-1}t_{n+1}^{j+1}}.\\
&&&&\\
(ii) & {\rm For}\ k \in \mathbb{Z}\setminus \mathbb{N}: & t_n^kg_{n+1} & = & q^{k}g_{n+1}^{-1} t_{n+1}^k \ + \ (q^{-1}-1)\underset{j=1}{\overset{-k}{\sum}}{q^{k+j}t_{n}^{-j}t_{n+1}^{k+j}}.\\
\end{array}
\]
\end{lemma}

\begin{proof}
We only prove relations (i) by induction on $k\in \mathbb{N}$. Relations (ii) follow similarly. For $k=1$ we have that $t_ng_{n+1} \ = \ g_{n+1}^{-1} t_{n+1}^k$ which holds. Assume now that the relation holds for $k$. Then, for $k+1$ we have:
\[
\begin{array}{lclc}
t_n^{k+1}g_{n+1} & = & t_n\cdot \underline{t_n^{k}g_{n+1}} \ \overset{ind.}{\underset{step}{=}}\ q^{-(k-1)}t_n\cdot \underline{g_{n+1}^{-1}} t_{n+1}^k\ +\ (q^{-1}-1)\sum_{j=0}^{k-2}{q^{-j}t_n\cdot t_{n}^{k-j-1}t_{n+1}^{j+1}} & =\\
&&&\\
& = & q^{-(k-1)}t_n\cdot \left(q^{-1}g_{n+1}+(q^{-1}-1)\right) t_{n+1}^k\ +\ (q^{-1}-1)\sum_{j=0}^{k-2}{q^{-j} t_{n}^{k-j}t_{n+1}^{j+1}} & =\\
&&&\\
& = & q^{-k}\underline{t_n\cdot g_{n+1}} t_{n+1}^k \ +\ q^{-(k-1)}(q^{-1}-1) t_n t_{n+1}^k\ +\ (q^{-1}-1)\sum_{j=0}^{k-2}{q^{-j} t_{n}^{k-j}t_{n+1}^{j+1}} & =\\
&&&\\
& = & q^{-k}g_{n+1}^{-1} t_{n+1}^{k+1}\ +\ (q^{-1}-1)\sum_{j=0}^{k-1}{q^{-j}t_{n}^{k-j}t_{n+1}^{j+1}}.&
\end{array}
\]
\end{proof}

\begin{nt}\rm
In what follows, for the expressions that we obtain after appropriate conjugations we shall use the notation $\widehat{=}$. We will also use the symbol $\simeq$ when a stabilization move is performed and $\widehat{\simeq}$ when both stabilization moves and conjugation are used. Note that by stabilization move we mean application of the third rule of the trace.
\end{nt}

It is worth mentioning that conjugation and stabilization moves are both captured by the properties of the trace tr, and, so, if $T \in \Lambda$ such that $T \ \widehat{\cong}\  \underset{i}{\sum}{a_i \tau_i}$ for some coefficients $a_i \in  \mathbb{C}\left[q^{\pm 1}, z^{\pm 1} \right]$ and $\tau_i \in \Lambda^{aug}$ for all $i$, then ${\rm tr}(T)\ =\ {\rm tr}(\underset{i}{\sum}{a_i \tau_i})\ =\ \underset{i}{\sum}{a_i {\rm tr}(\tau_i)}$. 


\smallbreak

We now present a key result needed for the proof of Theorem~\ref{mainre}.

\begin{lemma}\label{ordlem}
Applying conjugation and stabilization moves, elements in $\Lambda^{aug}$ can be expressed as sums of elements in $\Lambda^{aug}$ of lower order. Namely:
$$\Lambda^{aug}\ \ni\  \tau\ \widehat{\cong}\ \underset{i}{\sum} a_i\cdot \tau_i,\ {\rm where}\ \Lambda^{aug}\ \ni\ \tau_i < \tau\ \forall\ i. $$
\end{lemma}

\begin{proof}
Let $\tau\ =\ t^{k_0}t_1^{k_1}\ldots t_m^{k_m} \in \Lambda^{aug}_{(k)}$, where $k_{m-1}, k_m \in \mathbb{N}$. The case $k_{m-1}, k_m \in \mathbb{Z}$ follows similarly. Applying Eq.~\ref{l2} and Lemma~\ref{old2} we have that:

\[
\begin{array}{ccl}
t^{k_0}t_1^{k_1}\ldots \underline{t_m^{k_m}} & = & t^{k_0}t_1^{k_1}\ldots \underline{t_m^{k_m-1}g_m}t_{m-1}g_m\ \overset{(\ref{l2})}{\widehat{\cong}}\ (q-1)\; \underset{j=0}{\overset{k_m-2}{\sum}}q^j t^{k_0}\ldots t_{m-1}^{k_{m-1}+1-j}\underline{t_m^{k_m-1-j}\cdot g_m}\ +\\
&&\\
&+& q^{k_m-1}\; t^{k_0}\ldots \underline{t_{m-1}^{k_{m-1}}g_m}t_{m-1}^{k_m}g_m\ \overset{(\ref{l2})}{\underset{Lemma~\ref{old2}}{\widehat{\cong}}}\ (q-1)z\; \underset{j=0}{\overset{k_m-2}{\sum}}q^{k_m-1}t^{k_0}\ldots t_{m-1}^{k_{m-1}+k_m}\ +\\
&&\\
& + & (q-1)^2\; \underset{j=0}{\overset{k_m-2}{\sum}}\underset{\phi=0}{\overset{k_m-2-j}{\sum}}q^{j+\phi} t^{k_0}\ldots t_{m-1}^{k_{m-1}+1+j+\phi}t_m^{k_{m}-1-j-\phi}\ +\ \\
&&\\
& + & q^{k_m-k_{m-1}}\; t^{k_0}\ldots t_{m-1}^{k_m}t_m^{k_{m-1}}\ -\ q^{k_{m}-1}(q-1)z(k_{m-1}-1)\; t^{k_0}\ldots t_{m-1}^{k_{m-1}+k_m}\ - \\
&&\\
& - & q^{k_m-2}(q-1)^2\; \underset{\phi=0}{\overset{k_{m-1}-2}{\sum}}\underset{j=0}{\overset{\phi}{\sum}} q^{j-\phi} t^{k_0}\ldots t_{m-1}^{k_{m-1}-\phi-1+k_m+j}t_m^{\phi+1-j}.\\
\end{array}
\]

\noindent By Definition~\ref{order}, monomials on the right-hand side of the relation above are in $\Lambda^{aug}$ and they are of lower order than the initial monomial $\tau$. This concludes the proof.
\end{proof}

\begin{remark}\rm
Conjugation and stabilization moves that are used in the proof make Lemma~\ref{ordlem} valid on the level of the trace. It is worth mentioning though that on the level of the algebra ${\rm H}_{1, n}(q)$, Lemma~\ref{ordlem} is not valid. 
\end{remark}

\begin{lemma}\label{lem1}
The following relations hold for $k_0, k_1 \in \mathbb{Z}$:
\[
\begin{array}{llll}
t^{k_0}t_1^{k_1} & \widehat{\simeq} & q^{k_1-k_0}\cdot t^{k_1}t_1^{k_0} \ + \ q^{k_1-1} (q-1)\cdot z\cdot (k_1-k_0)\cdot t^{k_0+k_1}& +\\
&&&\\
& + & (q-1)^2 \ \underset{j=0}{\overset{k_1-2}{\sum}}\ \underset{\phi=0}{\overset{k_1-2-j}{\sum}}{q^{j+\phi}}t^{k_0+j+1+\phi}t_1^{k_1-1-j-\phi}& -\\
& -& (q-1)^2 \ \underset{j=0}{\overset{k_0-2}{\sum}}\ \underset{\phi=0}{\overset{j}{\sum}}{q^{k_1-j-2+\phi}}t^{k_0+k_1-j-1+\phi}t_1^{j+1-\phi}.& \\
\end{array}
\]
\end{lemma}

\begin{proof}
We only prove the case where $k_0, k_1 \in \mathbb{N}$ applying Eq.~\ref{l2} and Lemma~\ref{old2} on $t^{k_0}t_1^{k_1}$. All other cases for $k_0, k_1$ follow similarly.

\[
\begin{array}{llll}
t^{k_0}t_1^{k_1} & = & t^{k_0}\underline{t_1^{k_1-1}\cdot \sigma_1}t\sigma_1 \ \overset{(\ref{l2})}{=}\ t^{k_0}\cdot \left[ (q-1) \underset{j=0}{\overset{k_1-2}{\sum}}q^j\cdot t^jt_1^{k_1-1-j}\ +\ q^{k_1-1}\cdot \sigma_1t^{k_1-1}\right]\cdot t\sigma_1 & =\\
&&&\\
& = & \underset{A}{\underbrace{(q-1) \underset{j=0}{\overset{k_1-2}{\sum}}q^j\cdot t^{k_0+j+1}t_1^{k_1-1-j}\sigma_1}}\ +\ \underset{B}{\underbrace{q^{k_1-1}\cdot t^{k_0} \sigma_1t^{k_1}\sigma_1}}.\quad {\rm Moreover: } & \\
\end{array}
\]

$A\ =\  (q-1) \underset{j=0}{\overset{k_1-2}{\sum}}q^j\cdot t^{k_0+j+1}\underline{t_1^{k_1-1-j}\sigma_1} \ \overset{(\ref{l2})}{\widehat{\simeq}} $

\smallbreak

${\widehat{\simeq}}\ (q-1) \underset{j=0}{\overset{k_1-2}{\sum}}q^j\cdot t^{k_0+j+1}\cdot \left[ (q-1) \underset{\phi=0}{\overset{k_1-2-j}{\sum}}q^{\phi}\cdot t^{\phi}t_1^{k_1-1-j-\phi}\ +\ q^{k_1-1-j}\cdot \sigma_1 t^{k_1-1-j}  \right]\ \widehat{\simeq}$

\smallbreak

$\widehat{\simeq}\ (q-1)^2 \underset{j=0}{\overset{k_1-2}{\sum}} \underset{\phi=0}{\overset{k_1-2-j}{\sum}} q^{j+\phi}\cdot t^{k_0+j+1+\phi}{t_1^{k_1-1-j-\phi}} \ +\  (q-1)\cdot q^{k_1-1}\cdot z\cdot (k_1-1) \cdot t^{k_0+k_1}$

\smallbreak

\noindent and

\smallbreak

$B\ =\ q^{k_1-1}\cdot \underline{t^{k_0} \sigma_1} \cdot t^{k_1}\sigma_1 \ \overset{Lemma~\ref{old2}}{\widehat{\simeq}}\ q^{k_1-1}\cdot \left[ q^{-(k_0-1)}\sigma_1^{-1}t_1^{k_0}\ +\ \underset{j=0}{\overset{k_0-2}{\sum}} q^{-j}(q^{-1}-1)t^{k_0-j-1}t_1^{j+1} \right]\cdot t^{k_1} \sigma_1\ \widehat{\simeq}$

\smallbreak

$\widehat{\simeq}\ q^{k_1-k_0}\cdot t^{k_1}t_1^{k_0}\ +\ \underset{j=0}{\overset{k_0-2}{\sum}} q^{k_1-j-1}(q^{-1}-1)t^{k_0+k_1-j-1}\underline{t_1^{j+1}\sigma_1} \ \overset{(\ref{l2})}{\widehat{\simeq}}$

\smallbreak

$\widehat{\simeq}\ q^{k_1-k_0}\cdot t^{k_1}t_1^{k_0}\ +\ \underset{j=0}{\overset{k_0-2}{\sum}} q^{k_1-j-1}(q^{-1}-1)t^{k_0+k_1-j-1}\cdot \left[ (q-1)\cdot \underset{\phi=0}{\overset{j}{\sum}} q^{\phi} t_1^{j+1-\phi}\ +\ q^{j+1}\sigma_1t^{j+1} \right]\ \widehat{\simeq}$

$\widehat{\simeq}\  q^{k_1-k_0}\cdot t^{k_1}t_1^{k_0}\ +\ (q-1)(q^{-1}-1)\underset{j=0}{\overset{k_0-2}{\sum}} \underset{\phi=0}{\overset{j}{\sum}} q^{k_1-j-1+\phi}t^{k_0+k_1-j-1+\phi}\ +\ q^{k_1}(q^{-1}-1)\cdot z\cdot (k_0-1)\cdot t^{k_0+k_1}$.

\end{proof}

\subsection{The induction steps}

We shall first demonstrate Theorem~\ref{mainre} in the simplest non-trivial case, namely, monomials of index $1$, which comprise the simplest examples demonstrating the theorem. This will serve as the basis for the induction argument in the proof of Theorem~\ref{mainre}.

\begin{prop}\label{indba}
Let $k_0, k_1\in \mathbb{Z}$ and $t^{k_0}t_1^{k_1}\in \Lambda_{(k)}$. Then, the equations $X_{\widehat{t^{k_0}t_1^{k_1}}}=X_{\widehat{t^p t_1^{k_0}t_2^{k_1}\cdot \sigma_2\sigma_1^{\epsilon}\sigma_2^{-1}}}$ are linear combinations of equations of the form $X_{\widehat{t^{l_0}t_1^{l_1}}}=X_{\widehat{t^p t_1^{l_0}t_2^{l_1}\cdot \sigma_1^{\epsilon}}}$, where $t^{l_0}t_1^{l_1}\in \Lambda_{(k)}^{aug}$ and $t^{l_0}t_1^{l_1}< t^{k_0}t_1^{k_1}$ for all $l_0, l_1 \in \mathbb{Z}$. Furthermore, for $j>2$, the following diagram commutes:
\[
\begin{array}{cccc}
\Lambda_{(k)} \ni& t^{k_0}t_1^{k_1} & \overset{bbm_j}{\longrightarrow} & t^pt_1^{k_0}t_2^{k_1}\sigma_j\ldots \sigma_1^{\pm 1}\ldots \sigma_j^{-1} \\
& | && | \\
& conj.\ \& \ stab. & & conj.\ \& \ stab. \\
& \downarrow && \downarrow \\
\Lambda^{aug}_{(k)} \ni & \underset{i}{\sum}a_i t^{l_0}t_1^{l_1} & \overset{bbm_{1}}{\longrightarrow} & \underset{i}{\sum}a_i t^pt_1^{l_0}t_2^{l_1} \sigma_1^{\pm 1}\\
\end{array}
\]
\end{prop}

\begin{proof}

We only need to prove the case $j=2$ and we do that by considering $k_0, k_1 \in \mathbb{N}$. The case where both exponents are in $\mathbb{Z}\setminus \mathbb{N}$ and the case where one exponent is in $\mathbb{N}$ and the other in $\mathbb{Z}\backslash \mathbb{N}$ follow totally analogous.

\smallbreak

Consider the element $t^{k_0}t_1^{k_1}\in \Lambda_{(k)}$ and perform a bbm on its second moving strand. Then:
\[
\begin{array}{cccl}
t^{k_0}t_1^{k_1} & \overset{bbm_2}{\longrightarrow} & t^p t_1^{k_0}t_2^{k_1}\sigma_2\sigma_1^{\pm 1}\sigma_2^{-1} & {\rm and}\\ 
&&&\\
X_{\widehat{t^{k_0}t_1^{k_1}}} & = & X_{\widehat{t^p t_1^{k_0}t_2^{k_1}\sigma_2\sigma_1^{\pm 1}\sigma_2^{-1}}} & \Rightarrow\\
&&&\\
{\rm tr}\left(t^{k_0}t_1^{k_1} \right) & = & \Delta\cdot \left(\sqrt{\lambda } \right)^{a} {\rm tr}\left(t^p t_1^{k_0}t_2^{k_1}g_2g_1^{\pm 1}g_2^{-1} \right), &\\
\end{array}
\]
\noindent where $a=2\cdot (k_0+k_1)\pm 1\ =\ 2\cdot k \pm 1$.

\smallbreak

Consider now the elements $t^{l_0}t_1^{l_1}\in \Lambda^{aug}_{(k)}$ such that $t^{l_0}t_1^{l_1} < t^{k_0}t_1^{k_1}$ and perform a bbm on their first moving strand:
\[
\begin{array}{cccl}
t^{l_0}t_1^{l_1} & \overset{bbm_1}{\longrightarrow} & t^p t_1^{l_0}t_2^{l_1}\sigma_1^{\pm 1} & {\rm and}\\ 
&&&\\
X_{\widehat{t^{l_0}t_1^{l_1}}} & = & X_{\widehat{t^p t_1^{l_0}t_2^{l_1}\sigma_1^{\pm 1}}} &\Rightarrow \\
&&&\\
{\rm tr}\left(t^{l_0}t_1^{l_1} \right) & = & \Delta\cdot \left(\sqrt{\lambda } \right)^{a^{\prime}} {\rm tr}\left(t^p t_1^{l_0}t_2^{l_1}g_1^{\pm 1} \right), &\\
\end{array}
\]
\noindent where $a^{\prime}=2\cdot (l_0+l_1)\pm 1\ =\ 2\cdot k \pm 1$ again.

\bigbreak

\noindent It suffices to prove that if $t^{k_0}t_1^{k_1} \ \widehat{\simeq} \  \underset{i}{\sum} A_i\cdot t^{l_0}t_1^{l_1}$, for some coefficients $A_i$, then the following diagram commutes:
\[
\begin{array}{ccc}
t^{k_0}t_1^{k_1} & \widehat{\simeq} &  \underset{i}{\sum} A_i\cdot t^{l_0}t_1^{l_1}  \\
| & & | \\
 bbm_2  & & bbm_1 \\
\downarrow & & \downarrow\\
t^p t_1^{k_0}t_2^{k_1}\sigma_2\sigma_1^{\pm 1}\sigma_2^{-1} & \widehat{\simeq} & \underset{i}{\sum} A_i\cdot t^p t_1^{l_0}t_2^{l_1}\sigma_1^{\pm 1}\\
\end{array}
\]

On the one hand, from Lemma~\ref{lem1} we have:

\[
\begin{array}{llll}
t^{k_0}t_1^{k_1} & \widehat{\simeq} & q^{k_1-k_0}\cdot t^{k_1}t_1^{k_0} \ + \ q^{k_1-1} (q-1)\cdot z\cdot (k_1-k_0)\cdot t^{k_0+k_1}& +\\
&&&\\
& + & (q-1)^2 \ \underset{j=0}{\overset{k_1-2}{\sum}}\ \underset{\phi=0}{\overset{k_1-2-j}{\sum}}{q^{j+\phi}}t^{k_0+j+1+\phi}t_1^{k_1-1-j-\phi}& -\\
& -& (q-1)^2 \ \underset{j=0}{\overset{k_0-2}{\sum}}\ \underset{\phi=0}{\overset{j}{\sum}}{q^{k_1-j-2+\phi}}t^{k_0+k_1-j-1+\phi}t_1^{j+1-\phi}.& \\
\end{array}
\]

On the other hand, applying Eq.~\ref{l2} and Lemma~\ref{old2} on $t^p t_1^{k_0}t_2^{k_1}\sigma_2\sigma_1^{\pm 1}\sigma_2^{-1}$, we have that:

\noindent $ t^p t_1^{k_0}\underline{t_2^{k_1}\sigma_2}\sigma_1^{\pm 1}\sigma_2^{-1} \ \overset{(\ref{l2})}{\widehat{\simeq}}$
\[
\begin{array}{cl}
\widehat{\simeq} &  (q-1) \cdot  \underset{j=0}{\overset{k_1-1}{\sum}}q^j\cdot t^pt_1^{k_0+j}\underline{t_2^{k_1-j}\sigma_1^{\pm 1}}\sigma_2^{-1}\ +\ q^{k_1}\cdot t^p\underline{t_1^{k_0}\sigma_2} t_1^{k_1} \sigma_1^{\pm 1}\sigma_2^{-1}\ \overset{Lemma~\ref{old2}}{\widehat{\simeq}}\\
&\\
\widehat{\simeq} & (q-1) \cdot  \underset{j=0}{\overset{k_1-1}{\sum}}q^j\cdot t^pt_1^{k_0+j+1}\sigma_1^{\pm 1}\underline{t_2^{k_1-j-1}\sigma_2}\ +\ q^{k_1-k_0+1}\cdot \underline{t^p \sigma_2^{-1}}t_1^{k_1} t_2^{k_0} \sigma_1^{\pm 1}\sigma_2^{-1}\ +\\ 
&\\
 + & q^{k_1}\cdot \underset{j=0}{\overset{k_0-2}{\sum}}q^{-j} (q^{-1}-1)\cdot t^pt_1^{k_0+k_1-j-1}\underline{t_2^{j+1} \cdot \sigma_1^{\pm 1}\sigma_2^{-1}}\ \overset{(\ref{l2})}{\widehat{\simeq}}\\
&\\
\widehat{\simeq} & (q-1)^2\cdot \underset{j=0}{\overset{k_1-1}{\sum}} \underset{\phi=0}{\overset{k_1-j-2}{\sum}}q^{j+\phi} t^p t_1^{k_0+j+1+\phi} t_2^{k_1-j-1-\phi}\sigma_1^{\pm 1}\ + \\
&\\
+ & q^{k_1-1}(q-1)\cdot z\cdot k_1 \cdot t^p t_1^{k_0+k_1}\sigma_1^{\pm 1}\ -\ q^{k_1-k_0}(q-1) t^{p} t_1^{k_1+1} t_2^{k_0-1}\sigma_1^{\pm 1}\sigma_2^{-1}\ +\\
&\\
+ & q^{k_1-k_0} t^{p} t_1^{k_1} t_2^{k_0}\sigma_1^{\pm 1}\ +\ \underset{j=0}{\overset{k_0-2}{\sum}} q^{k_1-j}(q^{-1}-1) t^p t_1^{k_0-j+k_1}t_2^{j}\sigma_1^{\pm 1}\sigma_2\ \overset{(\ref{l2})}{\widehat{\simeq}}\\
&\\
=& (q-1)^2\cdot \underset{j=0}{\overset{k_1-1}{\sum}}\ \underset{\phi=0}{\overset{k_1-j-2}{\sum}}q^{j+\phi} t^p t_1^{k_0+j+1+\phi} t_2^{k_1-j-1-\phi}\sigma_1^{\pm 1}\ + \\
&\\
+ & q^{k_1-1}(q-1)\cdot z\cdot k_1 \cdot t^p t_1^{k_0+k_1}\sigma_1^{\pm 1}\ -\ \underset{C}{\underbrace{q^{k_1-k_0}(q-1)^2\underset{j=0}{\overset{k_0-2}{\sum}}q^j t^{p} t_1^{k_1+1+j} t_2^{k_0-1-j}\sigma_1^{\pm 1}}}\ +\\
&\\
+ & q^{k_1-1}(q-1)\cdot z\cdot t^pt_1^{k_0+k_1}\sigma_1^{\pm 1}\ -\ q^{k_1-1}(q-1)\cdot z\cdot (k_0-1)\cdot t^pt_1^{k_0+k_1}\sigma_1^{\pm 1}\ -\\
&\\
- &  \underset{D}{\underbrace{\underset{j=0}{\overset{k_0-2}{\sum}}\ \underset{\phi=0}{\overset{j-1}{\sum}}q^{k_1-j+\phi}(q^{-1}-1)(q-1)\cdot t^pt_1^{k_0+k_1-j+\phi}t_2^{j-\phi}\sigma_1^{\pm 1}}}\ =\\
&\\
= & (q-1)^2\cdot \underset{j=0}{\overset{k_1-1}{\sum}}\ \underset{\phi=0}{\overset{k_1-j-2}{\sum}}q^{j+\phi} t^p t_1^{k_0+j+1+\phi} t_2^{k_1-j-1-\phi}\sigma_1^{\pm 1}\ + \ q^{k_1-1}(q-1)\cdot z\cdot t^pt_1^{k_0+k_1}\sigma_1^{\pm 1}\ +\\
&\\
+ & q^{k_1-1}(q-1)\cdot z\cdot k_1 \cdot t^p t_1^{k_0+k_1}\sigma_1^{\pm 1}\ -\ q^{k_1-1}(q-1)\cdot z\cdot (k_0-1)\cdot t^pt_1^{k_0+k_1}\sigma_1^{\pm 1}\ -\\
&\\
- & \underset{C+D}{\underbrace{{(q-1)^2 \ \underset{j=0}{\overset{k_0-2}{\sum}}\ \underset{\phi=0}{\overset{j}{\sum}}{q^{k_1-j-2+\phi}}t^pt_1^{k_0+k_1-j-1+\phi}t_2^{j+1-\phi}\sigma_1^{\pm 1}}}}
\end{array}
\]

\noindent where in the last sumand we have set $j=k_0-r-2$. Thus, we have shown the following:

\[
\begin{array}{ccc}
t^{k_0}t_1^{k_1} & \overset{bbm_2}{\longrightarrow} & t^p t_1^{k_0}t_2^{k_1}\sigma_2\sigma_1^{\pm 1}\sigma_2^{-1} \\
&&\\
| & & | \\
conj.\ \& \ stab. & & conj.\ \& \ stab. \\ 
\downarrow & & \downarrow \\
&&\\
q^{k_1-k_0}\cdot t^{k_1}t_1^{k_0} & \overset{bbm_1}{\longrightarrow} & q^{k_1-k_0}\cdot t^pt_1^{k_1}t_2^{k_0}\sigma_1^{\pm 1}\\ 
+ & & +\\
q^{k_1-1} (q-1)\cdot z\cdot (k_1-k_0)\cdot t^{k_0+k_1} & \overset{bbm_1}{\longrightarrow} & q^{k_1-1} (q-1) \cdot z\cdot (k_1-k_0)\cdot t^pt_1^{k_0+k_1}\sigma_1^{\pm 1}\\
+ & & +\\
a \underset{j=0}{\overset{k_1-2}{\sum}}\ \underset{\phi=0}{\overset{k_1-2-j}{\sum}}{q^{j+\phi}}t^{k_0+j+1+\phi}t_1^{k_1-1-j-\phi} & \overset{bbm_1}{\longrightarrow} & a \underset{j=0}{\overset{k_1-2}{\sum}}\ \underset{\phi=0}{\overset{k_1-2-j}{\sum}}{q^{j+\phi}}t^pt_1^{k_0+j+1+\phi}t_2^{k_1-1-j-\phi}\sigma_1^{\pm 1}\\
- & & -\\
a \underset{j=0}{\overset{k_0-2}{\sum}}\ \underset{\phi=0}{\overset{j}{\sum}}{q^{k_1-j-2+\phi}}t^{k_0+k_1-j-1+\phi}t_1^{j+1-\phi} & \overset{bbm_1}{\longrightarrow} & a \underset{j=0}{\overset{k_0-2}{\sum}}\ \underset{\phi=0}{\overset{j}{\sum}}{q^{k_1-j-2+\phi}}t^pt_1^{k_0+k_1-j-1+\phi}t_2^{j+1-\phi}\sigma_1^{\pm 1} \\
\end{array}
\]

\noindent where $a=(q-1)^2$ and the proof is now concluded.
\end{proof}

Before proceeding with the proof of Theorem~\ref{mainre}, we shall first express elements $T\in \Lambda_{(k)} \subset \Lambda_{(k)}^{aug}$ as linear combinations of elements in $\Lambda^{aug}_{(k)}$, $T\ \widehat{\cong}\ \underset{i}{\sum}a_i\cdot T_i$, such that $T_i< T, \forall i$, and we will show that the following diagram commutes:

\[
\begin{array}{ccccc}
\Lambda_{(k)} &\ni & T & \overset{bbm_j}{\longrightarrow} & t^pT_{+}\sigma_j\ldots \sigma_2 \sigma_1^{\pm 1}\sigma^{-2}\ldots \sigma_j^{-1} \\
&& | && | \\
&& conj.\ \& \ stab. & & conj.\ \& \ stab. \\
&& \downarrow && \downarrow \\
\Lambda^{aug}_{(k)} &\ni & \underset{i}{\sum}a_i T_i & \overset{bbm_{j_i}}{\longrightarrow} & \underset{i}{\sum}a_i t^pT_{i_{+}} \sigma_{j_i}\ldots \sigma_2 \sigma_1^{\epsilon}\sigma^{-2}\ldots \sigma_{j_i}^{-1}\\
\end{array}
\]

\noindent where $j_i<j$ for all $i$. Namely, we prove the following result, which generalizes Proposition~\ref{indba}:

\begin{prop}\label{prop}
Let $T\in \Lambda_{(k)}$ for some $k\in \mathbb{Z}$. The equation $X_{\widehat{T}}\ = \ X_{\widehat{bbm_j(T)}}$ is a linear combination of the equations $X_{\widehat{T_i}}\ = \ X_{\widehat{bbm_{j_i}(T_i)}}$, where $T_i \in \Lambda^{aug}_{(k)}$, such that $T_i < T, \forall i$ and $j_i < j$, for all $i$.
\end{prop}

\begin{proof}

Let $T\in \Lambda_{(k)}$. Then, since $\Lambda_{(k)} \subset \Lambda_{(k)}^{aug}$, using conjugation and stabilization moves we have that 
$T\ \widehat{\cong}\ \underset{i}{\sum}a_i\cdot T_i$, such that $T_i \in \Lambda^{aug}_{(k)}$ and $T_i< T, \forall i$.

\smallbreak

\noindent Let $T\ =\ t^{k_0}\ldots t_m^{k_m}\in \Lambda_{(k)}$, such that $\underset{i=0}{\overset{m}{\sum}}k_i=k$. Applying a bbm on the $j^{th}$-moving strand of $T$ we obtain $bbm_j(T)=t^pt_1^{k_0}\ldots t_{m+1}^{k_m}\sigma_j\ldots \sigma_1^{\epsilon}\ldots \sigma_j^{-1}$, such that the sum of the exponents of the $\sigma_i$'s in the $bbm_j(T)$ is equal to $\underset{i=1}{\overset{m}{\sum}}2(i+1)\cdot k_i$. Similarly, if $T_i\ =\ t^{l_0}\ldots t_n^{l_n} \in \Lambda_{(k)}^{aug}$, then $\underset{i=0}{\overset{n}{\sum}}l_i=k$ and $bbm_{j_i}(T_i)=t^pt_1^{l_0}\ldots t_{n+1}^{l_n}\sigma_{j_i}\ldots \sigma_1^{\epsilon}\ldots \sigma_{j_i}^{-1}$, such that the sum of the exponents of the $\sigma_i$'s in the $bbm_{j_i}(T)$ is equal to $\underset{i=1}{\overset{n}{\sum}}2(i+1)\cdot l_i$. Thus, we have that:

\[
\begin{array}{cccccccr}
X_{\widehat{T}} & = & X_{\widehat{t^pT_{+}\sigma_m\ldots \sigma_1^{\epsilon}\ldots \sigma_m^{-1}}} & \Rightarrow &  {\rm tr}(T) & = & \Delta \sqrt{\lambda}^{2k}\cdot {\rm tr}(t^pT_{+}g_m\ldots g_1^{\epsilon}\ldots g_m^{-1}) & {\rm and }\\
X_{\widehat{T_i}} & = & X_{\widehat{t^pT_{i_{+}} \sigma_{j_i}\ldots \sigma_1^{\epsilon}\ldots \sigma_{j_i}^{-1}}} & \Rightarrow &  {\rm tr}(T_i) & = & \Delta \sqrt{\lambda}^{2k}\cdot {\rm tr}(t^pT_{i_+}g_{j_i}\ldots g_1^{\epsilon}\ldots g_{j_i}^{-1}).&\\
\end{array}
\]

\bigbreak

\noindent Since $T\ \widehat{\cong}\ \underset{i}{\sum}a_i\cdot T_i$, we have that ${\rm tr}(T)\ =\ \underset{i}{\sum}a_i\cdot {\rm tr}(T_i)$ and, thus, it suffices to prove that: 

$${\rm tr}(t^pT_{+}g_m\ldots g_1^{\epsilon}\ldots g_m^{-1})\ =\ \underset{i}{\sum}a_i\cdot {\rm tr}(t^pT_{i_+}g_{j_i}\ldots g_1^{\epsilon}\ldots g_{j_i}^{-1}).$$

\noindent This follows from Proposition~5 \& Theorem~8 \cite{DLP}, where it is shown that stabilization moves and braid band moves commute. The fact that conjugation and braid band moves do not commute, results in the need of performing bbm's on different moving strands. The fact that $j_i<j$ for all $i$ comes from the fact that stabilization moves and conjugation reduce the `tail' $w= g_m\ldots g_1^{\epsilon}\ldots g_m^{-1}$ in the Hecke algebra of type A, ${\rm H}_n(q)$, to monomials of the form $g_j\ldots g_1^{\epsilon}\ldots g_j^{-1}$, where $j<m$ \cite[Theorem~10]{DL2}. The proof is now concluded.
\end{proof}

\begin{remark}\label{rm2}\rm
The statements of Propositions~\ref{indba} and \ref{prop} are also valid for monomials in $\Lambda^{aug}$ and not just in the subset $\Lambda$. 
Namely, if $\tau \in \Lambda_{(k)}^{aug}$, then $\tau \ \widehat{\cong}\ \underset{i}{\sum} b_i \tau_i$, where $\tau_i \in \Lambda_{(k)}^{aug}$such that $\tau_i < \tau$ and $b_i \in \mathbb{C}[q^{\pm 1}, z^{\pm 1}]$ for all $i$. One can easily confirm this by observing that nowhere in the proofs the order in the exponents played any role.
\end{remark}

We are now in the position to complete the proof of Theorem~\ref{mainre}.

\subsection{Proof of the main result}



We prove Theorem~\ref{mainre} by strong induction on the order of elements in $\Lambda^{aug} \supset \Lambda$ with respect to the total ordering and also on the moving strand where the bbm is performed. The basis of the induction concerns the monomials of type $t^{k_0}t_1^{k_1}\in \Lambda$, which are of minimal order among all non-trivial monomials in $\Lambda^{aug}$. The performance of a bbm on the second moving strand of such a monomial reduces, by Proposition~\ref{indba}, to taking place on the first moving strand of related words of lower order.

Consider now a monomial $T \in \Lambda_{(k)}$. We assume that the statement of Theorem~\ref{mainre} is true for all elements $\tau \in \Lambda^{aug}_{(k)}$ of lower order than $T\in \Lambda_{(k)}$ and we will show that it is true for $T$. By Lemma~\ref{ordlem}, the monomial $T$ can be expressed as a sum of elements in $\Lambda^{aug}_{(k)}$ of lower order. Namely, $T\ \widehat{\cong}\ \underset{i}{\sum} a_i T_i$, where $T_i \in \Lambda_{(k)}^{aug}$ such that $T_i < T$ and  $a_i \in \mathbb{C}[q^{\pm 1}, z^{\pm 1}]$ for all $i$. By Remark~\ref{rm2}, Lemma~\ref{ordlem} is also valid for elements in $\Lambda^{aug}_{(k)}$. We distinguish the following cases:

\bigbreak

\begin{itemize}
\item[Case I:] A bbm is performed on the second moving strand of $T$. We shall prove that the theorem is valid for $T$, assuming it is valid for words of lower order in $\Lambda^{aug}$ with the bbm performed on their second moving strand. Namely, we prove that the following diagram commutes:

\[
\begin{array}{cccc}
\Lambda_{(k)} \ni & T & \overset{bbm_2}{\longrightarrow} & t^pT_{+}\sigma_2 \sigma_1^{\pm 1}\sigma_2^{-1} \\
& | && | \\
& conj.\ \& \ stab. & & conj.\ \& \ stab. \\
& \downarrow && \downarrow \\
\Lambda^{aug}_{(k)} \ni & \underset{i}{\sum}a_i T_i & \overset{bbm_{1}}{\longrightarrow} & \underset{i}{\sum}a_i t^pT_+ \sigma_1^{\pm 1}\\
\end{array}
\]
assuming that the diagram below commutes for all $\tau < T$:
\[
\begin{array}{cccc}
\Lambda^{aug}_{(k)} \ni & \tau & \overset{bbm_2}{\longrightarrow} & t^p\tau_{+}\sigma_2 \sigma_1^{\pm 1}\sigma_2^{-1} \\
& | && | \\
& conj.\ \& \ stab. & & conj.\ \& \ stab. \\
& \downarrow && \downarrow \\\Lambda^{aug}_{(k)} \ni & \underset{i}{\sum}b_i \tau_i & \overset{bbm_{1}}{\longrightarrow} & \underset{i}{\sum}b_i t^p\tau_{i_+} \sigma_1^{\pm 1}\\
\end{array}
\]
This follows directly from Proposition~\ref{prop} and the induction hypothesis. Indeed, by Proposition~\ref{prop} a $bbm_2$ on $T$ reduces on $bbm_1$'s and $bbm_2$'s on the $T_i$'s. For the $bbm_2$'s we apply the indiction hypothesis writing each $T_i$ as a sum: $T_i = \underset{j}\sum \tau_j$, where $\Lambda_{(k)}^{aug} \ni \tau_j < T_i$ for all $j$.

\bigbreak

\item[Case II:] A bbm is performed on any moving strand $m>2$ of $T$. We shall prove that the theorem is valid for $T$ assuming it is valid for monomials of lower order in $\Lambda_{(k)}^{aug}$ with the bbm performed on the $j^{th}$ moving strand, where $j<m$. Namely, we prove that the following diagram commutes:

\[
\begin{array}{cccc}
\Lambda_{(k)} \ni & T & \overset{bbm_m}{\longrightarrow} & t^pT_{+}\sigma_m \ldots \sigma_1^{\pm 1} \ldots \sigma_m^{-1} \\
& | && | \\
& conj.\ \& \ stab. & & conj.\ \& \ stab. \\
& \downarrow && \downarrow \\\Lambda^{aug}_{(k)} \ni & \underset{i}{\sum}a_i T_i & \overset{bbm_{1}}{\longrightarrow} & \underset{i}{\sum}a_i t^pT_{i_+} \sigma_1^{\pm 1}\\
\end{array}
\]
assuming that the diagram below commutes for all $\tau < T$ and for all $j<m$:
\[
\begin{array}{ccccr}
\Lambda^{aug}_{(k)} \ni & \tau & \overset{bbm_j}{\longrightarrow} & t^p\tau_{+}\sigma_j\ldots \sigma_1^{\pm 1} \ldots \sigma_j^{-1}& \\
& | && |&\\
& conj.\ \& \ stab. & & conj.\ \& \ stab. & \quad (*)\\
& \downarrow && \downarrow \\\Lambda^{aug}_{(k)} \ni & \underset{i}{\sum}b_i \tau_i & \overset{bbm_{1}}{\longrightarrow} & \underset{i}{\sum}b_i t^p\tau_{i_+} \sigma_1^{\pm 1}&\\
\end{array}
\]

\bigbreak

\noindent From Proposition~\ref{prop} we have that equations of the form $T\ \overset{bbm_m}{\longrightarrow} \ t^pT_{+}\sigma_m \ldots \sigma_1^{\pm 1} \ldots \sigma_m^{-1}$ can be expressed as sums of equations of the form $T_i\ \overset{bbm_j}{\longrightarrow} \ t^pT_{i_+}\sigma_j \ldots \sigma_1^{\pm 1} \ldots \sigma_j^{-1}$, where $j<m$ and $T_i < T$ for all $i$. Namely, we have the following diagram:

\[
\begin{array}{ccccc}
\Lambda_{(k)} \ni & T & \overset{bbm_m}{\longrightarrow} & t^p T_{+}\sigma_m \ldots \sigma_1^{\pm 1} \ldots \sigma_m^{-1} & (\diamondsuit) \\
& | && | &\\
& conj.\ \& \ stab. & & conj.\ \& \ stab.& \\
& \downarrow && \downarrow &\\
\Lambda^{aug}_{(k)} \ni & \underset{i}{\sum}a_i \; T_i & \overset{bbm_{1}}{\longrightarrow} & \underset{i_1}{\sum}a_{i_1} t^p T_{i_{1_+}} \sigma_1^{\pm 1} & (1) \\
& & \overset{bbm_{2}}{\searrow}& + & \\
& &  & \underset{i_2}{\sum}a_{i_2} \; t^p T_{i_{2_+}} \sigma_2\sigma_1^{\pm 1}\sigma_2^{-1} & (2) \\
& & \vdots & \vdots & \vdots \\
& & \overset{bbm_{m-1}}{\searrow} & + & \\
& &   & \underset{i_{m-1}}{\sum}a_{i_{m-1}}\; t^p T_{{i_{m-1}}_+} \sigma_{m-1} \ldots \sigma_1^{\pm 1} \ldots \sigma_{m-1}^{-1} & (m-1)\\
\end{array}
\]

\noindent Hence, Eq.~$(\diamondsuit)$ can be expressed as a combination of the Equations ($2$)---($m-1$). From the induction hypothesis now and Remark~\ref{rm2}, we have that Equations ($2$)---($m-1$) can be written as sums of equations obtained from elements in $\Lambda_{(k)}^{aug}$ of lower order, by performing bbm's on their first moving strand as assumed in the diagram $(*)$. This concludes the proof.
\end{itemize}

\hfill QED

\section{Conclusions}

In this paper we present an infinite system of equations, a solution of which corresponds to computing the HOMFLYPT skein module of the lens spaces $L(p, 1)$. The equations of this system are obtained by performing braid band moves on the first moving strand of elements in the set $\Lambda^{aug}$, which augments the basic set of the HOMFLYPT skein module of the solid torus, $\Lambda$. In \cite{DLP} a different infinite system of equations is presented, where bbm's are performed only on elements in $\Lambda$ but on any moving strand. Although $\Lambda \subset \Lambda^{aug}$, the main advantage of considering $\Lambda^{aug}$ is that we obtain more control on the infinite system of equations. In \cite{DL4} we elaborate on the solution of the infinite system.

\end{document}